\documentclass[12pt,reqno]{amsart}

\usepackage[latin1]{inputenc}
\usepackage{amsmath}
\usepackage{amsfonts}
\usepackage{amssymb}
\usepackage{graphics}
\usepackage{enumerate}
\usepackage{subfigure}
\usepackage{amssymb,amsmath,amsthm,amscd,epsf,latexsym,verbatim,graphicx,amsfonts,marvosym}
\input epsf.tex

\usepackage{amscd}
\usepackage{amsmath}
\usepackage{amssymb}
\usepackage{amsthm}
\usepackage{epsf}
\usepackage{latexsym}
\usepackage{verbatim}

\theoremstyle{plain}

\newtheorem{theorem}{Theorem}[section]

\newtheorem{prop}[theorem]{Proposition}
\theoremstyle{definition}

\theoremstyle{remark}

\newcommand{\bbC}{\mathbb{C}}
\newcommand{\bbD}{\mathbb{D}}
\newcommand{\bbN}{\mathbb{N}}

\newcommand{\eitheta}{e^{i\theta}}

\title[]{An Electrostatic Interpretation of the Zeros of Paraorthogonal Polynomials on the Unit Circle}
\author[]{Brian Simanek}
\date{}

\begin{document}
\maketitle

\begin{abstract}
We show that if $\mu$ is a probability measure with infinite support on the unit circle having no singular component and a differentiable weight, then the corresponding paraorthogonal polynomial $\Phi_n(z;\beta)$ solves an explicit second order linear differential equation.  We also show that if $\tau\neq\beta$, then the pair $(\Phi_n(z;\beta),\Phi_n(z;\tau))$ solves an explicit first order linear system of differential equations.  One can use these differential equations to deduce that the zeros of every paraorthogonal polynomial mark the locations of a set of particles that are in electrostatic equilibrium with respect to a particular external field.
\end{abstract}

\vspace{4mm}


\vspace{2mm}


\vspace{2mm}

\normalsize

\section{Introduction}\label{intro}

Given a positive probability measure $\mu$ with compact and infinite support in the complex plane $\bbC$, one can constuct the corresponding sequence of orthonormal polynomials $\{\varphi_n(z;\mu)\}_{n\geq0}$ where $\varphi_n(z;\mu)$ is a polynomial of degree exactly $n$ having positive leading coefficient, which we denote by $\kappa_n(\mu)$.  If we divide $\varphi_n(z;\mu)$ by $\kappa_n(\mu)$, then we obtain a monic degree $n$ polynomial, which we denote by $\Phi_n(z;\mu)$.  We will often suppress the $\mu$ dependence in our notation if there is no possibility for confusion.  Our study of these objects is motivated by the fact that orthogonal polynomials have proven to be valuable tools in the study of physical models.

The most well-studied collections of orthogonal polynomials come from measures of orthogonality supported on the real line and are known as orthogonal polynomials on the real line (OPRL).  Such polynomials have a variety of applications in spectral theory (see \cite{Rice}), potential theory (see \cite{SaffTot}), and the theory of special functions (see \cite{IsmBook}).  One of the key features of OPRL is the three-term recurrence formula satisfied by the orthonormal polynomials.  Many well-known families of OPRL (including the Hermite polynomials, the Laguerre polynomials, and the Jacobi polynomials) are polynomial solutions to particular families of linear second order differential equations with polynomial coefficients.  This property makes these particular families of polynomials especially applicable as we will discuss later in more detail.

A second collection of well-studied orthogonal polynomials comes from measures of orthogonality supported on the unit circle and are known as orthogonal polynomials on the unit circle (OPUC).  Just as OPRL has a close connection with self-adjoint operators, OPUC has a close connection with unitary operators (see \cite{Rice}).  In analogy with the three-term recurrence satisfied by the OPRL, the monic OPUC satisfy the Szeg\H{o} recursion:
\[
\Phi_{n+1}(z;\mu)=z\Phi_n(z;\mu)-\bar{\alpha}_n\Phi_n^*(z;\mu),
\]
where $\alpha_n\in\bbD:=\{z:|z|<1\}$ and $\Phi_n^*(z)=z^n\overline{\Phi_n(1/\bar{z})}$.  Therefore, to each infinitely supported probability measure $\mu$ on the unit circle, we can associate the sequence $\{\alpha_n\}_{n\geq0}$ of \textit{Verblunsky coefficients}.  A theorem of Verblunsky states that such a sequence also determines an infinitely supported probability measure (see \cite[Chapter 1]{OPUC1}).  The utility and applicability of the Hermite, Laguerre, and Jacobi polynomials inspired interest in families of OPUC that are solutions to linear second order differential equations.  To this end, Ismail and Witte proved the following result in \cite{IW}:

\begin{theorem}[Ismail $\&$ Witte, 2001]
Let $w(z)=e^{-v(z)}$ be differentiable in a neighborhood of the unit circle, have moments of all integral orders, and assume that the integrals
\[
\int_{|\zeta|=1}\frac{v'(z)-v'(\zeta)}{z-\zeta}\zeta^nw(\zeta)\frac{d\zeta}{i\zeta}
\]
exist for all integers $n$.  Then the corresponding orthonormal polynomials satisfy the differential relation
\begin{align*}
\varphi_n'(z)&=n\frac{\kappa_{n-1}}{\kappa_n}\varphi_{n-1}(z)-i\varphi_n^*(z)\int_{|\zeta|=1}\frac{v'(z)-v'(\zeta)}{z-\zeta}\varphi_n(\zeta)\overline{\varphi_n^*(\zeta)}w(\zeta)d\zeta\\
&\qquad\qquad+i\varphi_n(z)\int_{|\zeta|=1}\frac{v'(z)-v'(\zeta)}{z-\zeta}\varphi_n(\zeta)\overline{\varphi_n(\zeta)}w(\zeta)d\zeta.
\end{align*}
\end{theorem}

\noindent From this result, one can derive a second order differential equation to which $\varphi_n(z)$ is a solution; we refer the reader to \cite{IW} for details.
The main purpose of our investigation is to derive a comparable result for a related class of polynomials called \textit{paraorthogonal polynomials on the unit circle} (POPUC).

Given an infinitely supported probability measure $\mu$ on the unit circle and a complex number $\beta$ of modulus $1$, we define the paraorthogonal polynomial $\Phi_n(z;\beta;\mu)$ as the monic degree $n$ polynomial given by
\begin{align}\label{popucdef}
\Phi_n(z;\beta;\mu):=z\Phi_{n-1}(z;\mu)-\bar{\beta}\Phi_{n-1}^*(z;\mu).
\end{align}
Paraorthogonal polynomials and their zeros have received considerable recent attention from the research community (see for example \cite{DY,Gol,NJT,KiSt,FineIV,SimaGap,RankOne,Stoiciu,Wong}).  It is well-known and easy to show that all of the zeros of $\Phi_n(z;\beta;\mu)$ are distinct and lie on the unit circle.  Furthermore, if $\tau\neq\beta$ are distinct complex numbers of modulus $1$, then the zeros of $\Phi_n(z;\beta;\mu)$ and $\Phi_n(z;\tau;\mu)$ strictly interlace on the unit circle in that if $x$ and $y$ are two zeros of $\Phi_n(z;\beta;\mu)$ and $[x,y]$ is the arc of the unit circle that runs from $x$ to $y$ in the counter-clockwise direction, then $[x,y]\setminus\{x,y\}$ contains a zero of $\Phi_n(z;\tau;\mu)$.  Although paraorthogonal polynomials are not orthogonal polynomials, they often serve as an appropriate analog of OPRL in settings where the real line is replaced by the unit circle (see for example \cite{KiSt,Stoiciu}).  One of our main results is the next theorem, which can be thought of as an analog of the Ismail and Witte result that applies to paraorthogonal polynomials.  In fact, the theorem applies to any degree $n$ monic polynomial satisfying the relation (\ref{popucdef}) for some complex number $\beta$.  As in \cite{OPUC1}, we write a measure on the unit circle as a measure in the variable $\theta\in[0,2\pi)$.  For brevity and because there is no possibility for confusion, we will suppress the $\mu$-dependence of various quantities in our notation.  

\begin{theorem}\label{mainode}
Suppose $d\mu(\theta)=w(\theta)\frac{d\theta}{2\pi}$ is a probability measure on the unit circle, where $w$ is continuous on $[0,2\pi]$ (mod $2\pi$) and differentiable on $(0,2\pi)$ and let $\{\alpha_n\}_{n=0}^{\infty}$ be the corresponding sequence of Verblunsky coefficients.  If $\beta\in\bbC$, then the polynomial $y(z)=\Phi_n(z;\beta)$ defined by (\ref{popucdef}) solves the following differential equation on any domain including infinity or zero on which the coefficients are meromorphic:
\begin{align}\label{keyode}
&0=y''(z)+\bigg[\frac{1-n}{z}-\frac{h_n'(z;\beta;\beta)}{h_n(z;\beta;\beta)}\bigg]y'(z)\\
\nonumber&+\bigg[\frac{W[h_n(z;\beta;\beta),h_n(z;-\beta;\beta)]}{2\bar{\beta}zh_n(z;\beta;\beta)}-\frac{1}{z}\left((n+zG_n(z))G_n(z)+J_n(z)(D_n(z)-n\alpha_{n-1})\right)\bigg]y(z),
\end{align}
where
\begin{align*}
G_n(z):&=i\int_0^{2\pi}\frac{|\varphi_{n-1}^*(\eitheta)|^2w'(\theta)}{(z-\eitheta)}\frac{d\theta}{2\pi}\\
D_n(z):&=-iz\int_0^{2\pi}\frac{\varphi_{n-1}^*(\eitheta)^2w'(\theta)}{(z-\eitheta)e^{in\theta}}\frac{d\theta}{2\pi}\\
J_n(z):&=i\int_0^{2\pi}\frac{\varphi_{n-1}(\eitheta)^2w'(\theta)}{(z-\eitheta)e^{i(n-2)\theta}}\frac{d\theta}{2\pi}\\
h_n(z;x;y):&=\bar{x}(n(1-\bar{y}\alpha_{n-1})+zG_n(z)+\bar{y}D_n(z))-z(J_n(z)-\bar{y}G_n(z)),
\end{align*}
and  $W[f,g]$ denotes the Wronskian of $f$ and $g$ (that is, $W[f,g]=fg'-gf'$).
\end{theorem}

By allowing the parameter $\beta$ to be arbitrarily chosen in $\bbC$, we can apply our results to a wide variety of polynomials including paraorthogonal polynomials and perturbations of orthogonal polynomials without a precise understanding of the induced perturbation to the measure of orthogonality.

In our applications of Theorem \ref{mainode}, we will focus on the case $|\beta|=1$ so that the polynomial $\Phi_n(z;\beta;\mu)$ is a paraorthogonal polynomial.  Notice that if we set $\beta=\alpha_{n-1}$, then $\Phi_n(z;\beta;\mu)$ is just the monic orthogonal polynomial $\Phi_n(z;\mu)$.  In this case, Theorem \ref{mainode} yields a second order differential equation solved by $\Phi_n(z;\mu)$.  If we set $\beta=0$, then Theorem \ref{mainode} yields a second order differential equation solved by $z\Phi_{n-1}(z;\mu)$, which can be rewritten as a differential equation solved by $\Phi_{n-1}(z;\mu)$.  There is no reason to think that the second order differential equations solved by $\Phi_n(z;\mu)$ derived by these two methods will be the same as each other or the same as the one produced by the results in \cite{IW} if those results apply.  In fact, we will see by example in Section \ref{sntm} that these methods may yield different differential equations.  It is easy to see that the polynomial $\Phi_n(z;\mu)$ satisfies many differential equations if one allows the complexity of the coefficients to grow with $n$.  Indeed, if
\[
\Phi_n''(z;\mu)+P(z)\Phi_n'(z;\mu)+Q(z)\Phi_n(z;\mu)=0,
\]
and $R(z)$ is any entire meromorphic function, then
\[
\Phi_n''(z;\mu)+\Phi_n'(z;\mu)\left(P(z)+R(z)\Phi_n(z;\mu)\right)+\Phi_n(z;\mu)\left(Q(z)-R(z)\Phi_n'(z;\mu)\right)=0.
\]
In the next section we will discuss applications of Theorem \ref{mainode} to electrostatics.  In that context, the coefficients in the second order differential equation determine the location of charges that generate an electric field and the zeros of $\Phi_n(z;\beta;\mu)$ mark the location of charges in equilibrium with respect to that field.  It is not surprising that a particular configuration of charges can be in equilibrium with respect to different electric fields, so from a physical perspective, the lack of uniqueness of the differential equation is also expected.

\medskip

An intermediate step in the proof of Theorem \ref{mainode} is of interest in its own right and can be stated as follows:

\begin{theorem}\label{sysodes}
Suppose $\mu$ is as in Theorem \ref{mainode} and $\tau,\beta$ are distinct complex numbers.  The polynomials $u(z):=\Phi_n(z;\beta,\mu)$ and $v(z):=\Phi_n(z;\tau,\mu)$ solve the following system of differential equations on any domain containing infinity or zero on which the coefficients are meromorphic:
\begin{align}
\label{sys1thm}u'(z)&=v(z)\left(\frac{h_n(z;\beta;\beta)}{z(\bar{\beta}-\bar{\tau})}\right)-u(z)\left(\frac{h_n(z;\tau;\beta)}{z(\bar{\beta}-\bar{\tau})}\right)\\
\label{sys2thm}v'(z)&=v(z)\left(\frac{h_n(z;\beta;\tau)}{z(\bar{\beta}-\bar{\tau})}\right)-u(z)\left(\frac{h_n(z;\tau;\tau)}{z(\bar{\beta}-\bar{\tau})}\right).
\end{align}
\end{theorem}

The main restriction in the applicability of Theorems \ref{mainode} and \ref{sysodes} is the requirement that the measure be given by a continuous and differentiable weight function.  However, for any fixed $n\in\bbN$ and $\beta\in\partial\bbD$, the map from measures to degree $n$ paraorthogonal polynomials that map $0$ to $\-\bar{\beta}$ is far from injective.  In fact, in the pre-image of any such paraorthogonal polynomial is a measure that satisfies the hypotheses of the above theorems.  Indeed, we have the following result, which is a consequence of the Bernstein-Szeg\H{o} Theorem (see \cite[Theorem 1.7.8]{OPUC1})

\begin{theorem}\label{BZ}
Let $\mu$ be a probability measure with infinite support on the unit circle and suppose $\beta\in\partial\bbD$.  Then
\[
\Phi_n(z;\beta;\mu)=\Phi_n\left(z;\beta;\frac{e^{i(n-1)\theta}}{\varphi_{n-1}(\eitheta)\varphi_{n-1}^*(\eitheta)}\,\frac{d\theta}{2\pi}\right)
\]
where $\varphi_{n-1}(z)=\varphi_{n-1}(z;\mu)$.
\end{theorem}

The following result will be relevant for our applications and is especially useful when combined with Theorem \ref{BZ}.

\begin{prop}\label{intod}
Suppose $w$ is as in Theorem \ref{mainode}.
\begin{itemize}
\item[i)] If $w'(\theta)=f(\eitheta)$ where $f$ is analytic in a neighborhood of the unit circle, then the functions $G_n$, $D_n$, and $J_n$ (defined in Theorem \ref{mainode}) can be analytically continued from $\{z:|z|>1\}$ to $\{z:|z|>r\}$ for some $r<1$.
\item[ii)] If $w'(\theta)=f(\eitheta)$ where $f$ is a rational function without poles on the unit circle, then the functions $G_n$, $D_n$, and $J_n$ (defined in Theorem \ref{mainode}) can be meromorphically continued from $\{z:|z|>1\}$ to all of $\bbC$ as rational functions with no poles except possibly at zero and the poles of $f$ inside $\bbD$.
\end{itemize}
\end{prop}

\begin{proof}
i) We present the proof for $G_n$; the proof in the other two cases is even easier.  We will make use of the fact that
\[
|\varphi_n^*(\eitheta)|^2=\varphi_n^*(\eitheta)\varphi_n(\eitheta)e^{-in\theta}.
\]
Therefore, we can write
\begin{align}\label{gint}
G_n(z)=i\int_0^{2\pi}\frac{\varphi_{n-1}^*(\eitheta)\varphi_{n-1}(\eitheta)f(\eitheta)}{(z-\eitheta)e^{in\theta}}\frac{d\theta}{2\pi}=\int_{|\zeta|=1}\frac{\varphi_{n-1}^*(\zeta)\varphi_{n-1}(\zeta)f(\zeta)}{(z-\zeta)\zeta^{n+1}}\frac{d\zeta}{2\pi}.
\end{align}
If $|z|>1$, then our hypotheses allow us to move the contour of integration to the circle $\{\zeta:|\zeta|=r\}$ for some $r<1$ and hence we obtain an analytic continuation of $G_n(z)$ to the exterior of this circle.

ii)  Given the form of $G_n$, $D_n$, and $J_n$, it suffices to show that if $R(t)$ is a rational function of $t$ without poles on the unit circle, then when $|z|>1$
\[
\int_0^{2\pi}\frac{R(\eitheta)}{z-\eitheta}\,\frac{d\theta}{2\pi}
\]
is a rational function of $z$ with no poles except possibly at $0$ and the poles of $R$ inside $\bbD$.  By using the partial fraction decomposition of $R$, this follows from the fact that if $P$ is a polynomial and $m$ is a non-negative integer, then
\begin{align*}
\int_0^{2\pi}\frac{P(\eitheta)}{z-\eitheta}\,\frac{d\theta}{2\pi}&=\frac{P(0)}{z},\qquad\qquad |z|>1,\\
\int_0^{2\pi}\frac{1}{(z-\eitheta)(\eitheta-x)^m}\,\frac{d\theta}{2\pi}&=\begin{cases}
\frac{1}{z(-x)^m},\qquad\qquad &|z|>1,\,|x|>1,\\
\frac{1}{z(z-x)^m},\qquad &|z|>1,\,|x|<1.
\end{cases}
\end{align*}
\end{proof}

By combining the previous two results, we conclude that every paraorthogonal polynomial is a solution on all of $\bbC$ to a differential equation of the form (\ref{keyode}) that has rational coefficients.  This will be especially relevant for our applications.

In the next section, we will discuss an application of polynomial solutions to linear second order differential equations.  In Section \ref{prf} we will prove Theorems \ref{mainode} and \ref{sysodes}.  The key idea will be to use the Szeg\H{o} recursion in several places to simplify various formulas.  Finally, in Section \ref{exa} we present some detailed examples that highlight the applications discussed in the next section.

\vspace{2mm}

\noindent\textbf{Acknowledgments.}  We would like to thank Mihai Stoiciu for much useful conversation and also Andrei Martinez-Finkelshtein for useful feedback on this work and for bringing the paper \cite{BMR} to our attention.  We would also like to thank the anonymous referee for directing us to the observation that Theorem \ref{mainode} can be applied even if $|\beta|\neq1$.

\section{Applications}\label{apps}

In this section, we will highlight some important consequences and applications of the main results of the previous section.

\subsection{Electrostatics}\label{elec}

Orthogonal polynomials have proven to be useful tools when studying the equilibrium positions of electrons that repel each other via the two-dimensional Coulomb interaction.   It was Stieltjes work that lead to the discovery that if $n\geq3$ particles are confined to the interval $[-1,1]$ interacting by means of a logarithmic (i.e. Coulomb) potential, then the unique energy minimizing configuration consists of one particle located at $1$, one located at $-1$, and the other $n-2$ particles located at the (distinct) zeros of a particular Jacobi polynomial (see \cite{Stieltjes,Stieltjes1,Stieltjes2} and see \cite{Szeg} for a proof).  The key step in the proof of this fact is rewriting the equilibrium condition on a collection of points as a second order linear differential equation that is solved by the polynomial with zeros at precisely those points and then recognizing that a Jacobi polynomial with the appropriate choice of parameters is a solution to this differential equation (see \cite{Szeg} for details).  Further developments in the electrostatic interpretation of zeros of orthogonal polynomials on the real line can be found in \cite{DL,DVA,Grun1,Ism1,Ism2,MMFMG}.

Our main application of the results in Section \ref{intro} will be to demonstrate that the zeros of certain families of POPUC are the locations of points that are in electrostatic equilibrium.  We will use the convention (as in \cite{Grin}) that if a particle of charge $q$ is located at a point $a\in\bbC$ and a particle of charge $p$ is located at a point $b\in\bbC$, then the force on the particle at $b$ due to the particle at $a$ is $2pq/(\bar{b}-\bar{a})$.

We will consider the problem of creating an electric field that will keep identical charges at fixed points on the unit circle stationary.  More precisely, suppose we are given a collection $\{x_1,\ldots,x_n\}\subseteq\partial\bbD$.  We will demonstrate a way to find a number $m$, a collection of points $\{a_i\}_{i=1}^m\subseteq\bbC\setminus\{x_1,\ldots,x_n\}$, and a set of real charges $\{q_i\}_{i=1}^m$ so that if a particle of charge $+1$ is placed placed at each $x_j$ ($j=1,\ldots,n$) and a particle of charge $q_i$ is placed at $a_i$ ($i=1,\ldots,m$), then the total force on the particle at each $x_j$ is zero ($j=1,\ldots,n$).  In other words, if we have a collection of identically charged particles all lying on a concentric circle, we will demonstrate a way to construct an electric field that will keep these particles stationary.  The points $\{a_i\}_{i=1}^m$ and charges $\{q_i\}_{i=1}^m$ comprise what we call a \textit{set of electric field generators} (in the language of \cite{FW}, the charged particles at $\{x_j\}_{j=1}^n$ would be called \textit{mobile charges} and the charged particles at $\{a_i\}_{i=1}^m$ would be called \textit{impurity charges}).  It is important to keep in mind that the charged particles at the points $\{x_j\}_{j=1}^n$ interact with each other as well as with the electric field generators.

With our motivation now clearly stated, we provide the following definitions:

\vspace{2mm}

\noindent\textbf{Definition.}  i) Given a set of electric field generators $\{a_i\}_{i=1}^m$ and $\{q_i\}_{i=1}^m$, we will say that a collection of points $\{x_1,\ldots,x_n\}$ located on a smooth curve $\Gamma$ is in \textit{$\Gamma$-normal electrostatic equilibrium} if for each $j=1,\ldots,n$, the force at $x_j$ is normal to $\Gamma$ at $x_j$.

ii) Given a set of electric fields generators $\{a_i\}_{i=1}^m$ and $\{q_i\}_{i=1}^m$, we will say that a collection of points $\{x_1,\ldots,x_n\}\subseteq\partial\bbD$ is in \textit{total electrostatic equilibrium} if for each $j=1,\ldots,n$,
\begin{align}\label{none}
\sum_{{k=1}\atop{k\neq j}}^n\frac{1}{x_j-x_k}+\sum_{i=1}^{m}\frac{q_i}{x_j-a_{i}}=0.
\end{align}

\vspace{2mm}

Notice that if $\Gamma$ is the unit circle, then the $\Gamma$-normal electrostatic equilibrium condition can be rewritten as
\begin{align}\label{radial}
\mbox{Im}\left[x_j\left(\sum_{{k=1}\atop{k\neq j}}^n\frac{1}{x_j-x_k}+\sum_{i=1}^{m}\frac{q_i}{x_j-a_{i}}\right)\right]=0,\qquad\qquad j=1,2,\ldots,n,
\end{align}
so it is clear that a collection of points on the unit circle that is in total electrostatic equilibrium is also in $\partial\bbD$-normal electrostatic equilibrium, but the converse is false.  Indeed, it is well-known and easy to show that $n$ particles of identical non-zero charge placed at the $n^{th}$ roots of unity and subject to no external force are in $\partial\bbD$-normal electrostatic equilibrium, but are not in total electrostatic equilibrium.

The formula (\ref{radial}) is a restatement of the fact that if we place particles of charge $+1$ at each $x_j$ ($j=1,\ldots,n$) and a particle of charge $q_i$ at each $a_i$ ($i=1,\ldots,m$), then the condition (\ref{radial}) is satisfied if and only if the force exerted on the particle at $x_j$ is normal to the unit circle at $x_j$ for each $j=1,\ldots,n$ (see \cite{Grin,MFMGO}).  Similarly, the condition (\ref{none}) is satisfied if and only if the force exerted on the particle at $x_j$ is equal to zero for each $j=1,\ldots,n$ (see \cite{MBA}).

As in the case of the interval, orthogonal polynomials are a useful tool when studying electrostatic equilibria on the circle.  In \cite{FR}, Forrester and Rogers use Jacobi polynomials to find a collection of $2n$ points on the unit circle that is in $\partial\bbD$-normal electrostatic equilibrium when the electric field generators consist of a particle with charge $p$ at $1$ and a particle with charge $q$ at $-1$.  Their result assumes the added symmetry condition that each arc of the unit circle connecting $1$ to $-1$ contains an equal number of points.

Our main application is the following result:

\begin{theorem}\label{bigelec}
Given any collection of $n\geq2$ distinct points $\{x_1,\ldots,x_n\}\subseteq\partial\bbD$, there exists a set of electric field generators so that the collection $\{x_1,\ldots,x_n\}$ is in total electrostatic equilibrium.
\end{theorem}

In fact, given any $n$ distinct points $\{x_1,\ldots,x_n\}\subseteq\partial\bbD$, we can write down an explicit algorithm for finding the electric field generators.  We proceed as follows:

\begin{itemize}
\item \underline{Step 1}:  Define the measure $\mu_n$ on $\partial\bbD$ by
\[
\mu_n=\frac{1}{n}\sum_{j=1}^n\delta_{x_j},
\]
and define $\beta:=(-1)^{n+1}\prod_{j=1}^n\bar{x}_j$.
\item \underline{Step 2}: Perform Gram-Schmidt orthogonalization on the linearly independent set $\{1,z,\ldots,z^{n-1}\}$ in $L^2(\partial\bbD,\mu_n)$ to arrive at the sequence of orthonormal polynomials $\{1,\varphi_1(z;\mu_n),\ldots,\varphi_{n-1}(z;\mu_n)\}$.
\item \underline{Step 3}:  Define the probability measure
\[
d\nu_n:=\frac{1}{|\varphi_{n-1}(\eitheta;\mu_n)|^2}\frac{d\theta}{2\pi}.
\]
\item \underline{Step 4}:  Calculate the quantity $h_n(z;\beta;\beta)$ for the measure $\nu_n$ in the domain $\{z:|z|>1\}$.  By Proposition \ref{intod}, it will be a rational function $S_1(z)/S_2(z)$ for some polynomials $S_1$ and $S_2$.
\item \underline{Step 5}:  Place a particle of charge $-1/2$ at each zero of $S_1$, a particle of charge $+1/2$ at each zero of $S_2$, and a particle of charge $\frac{1}{2}(1-n)$ at zero.
\end{itemize}

The validity of the above algorithm will follow from the proof of Theorem \ref{bigelec}.  In order to prove Theorem \ref{bigelec}, we need to translate the equilibrium problem into a differential equation so that we may apply the results of Section \ref{intro}.  The appropriate differential equation in this setting is called a Lam\'{e} equation, which we now discuss.

\subsection{The Lam\'{e} Equation}\label{lame}

A generalized Lam\'{e} differential equation is a differential equation of the form
\begin{align}\label{lameode}
y''(z)+\left(\sum_{i=1}^m\frac{t_i}{z-w_i}\right)y'(z)+\frac{S(z)}{\prod_{j=1}^m(z-w_j)}y(z)=0,
\end{align}
where $S(z)$ is a polynomial of degree at most $m-2$.  Suppose we can find a polynomial $P(z)$ of degree $N$ with distinct zeros that solves (\ref{lameode}) and also satisfies $P(w_i)\neq0$ for $i=1,\ldots,m$.  Let $\{p_j\}_{j=1}^N$ be the zeros of $P$.  It is easy to verify that
\[
\frac{P''(p_j)}{P'(p_j)}=\sum_{{k=1}\atop{k\neq j}}^N\frac{2}{p_j-p_k},
\]
so setting $y=P$ and $z=p_j$ in (\ref{lameode}) shows
\[
\sum_{{k=1}\atop{k\neq j}}^N\frac{1}{p_j-p_k}+\sum_{i=1}^m\frac{t_i/2}{p_j-w_i}=0\qquad\qquad j=1,\ldots,N.
\]
We recognize this as the condition for total electrostatic equilibrium with external field generated by a collection of charged particles located at $\{w_i\}_{i=1}^m$, where the particle at $w_i$ carries charge $t_i/2$.  We will see that the proof of Theorem \ref{bigelec} follows from applying the above reasoning to the differential equation in Theorem \ref{mainode} under the appropriate hypotheses.  To this end, Proposition \ref{intod} will be essential.

There is an extensive literature on the topic of generalized Lam\'{e} differential equations and their relevance to electrostatics.  We refer the reader to the references \cite{DVA,MMFMG,Marden,MFMGO,MBA,SV,Shap} for further information.  We also refer the reader to \cite{DL,EKLW,GKM1,GKM2} for further results concerning applications of polynomial solutions to second order differential equations.

\section{Proofs and Calculations}\label{prf}

\begin{proof}[Proof of Theorems \ref{mainode} and \ref{sysodes}]
Our proof begins very much in the same spirit as the proof of \cite[Theorem 2.1]{IW}.  We begin by writing
\begin{align*}
\Phi_n'(z;\beta)&=\sum_{k=0}^{n-1}\varphi_k(z)\int_0^{2\pi}\Phi_n'(\eitheta;\beta)\overline{\varphi_k(\eitheta)}w(\theta)\frac{d\theta}{2\pi}\\
&=\sum_{k=0}^{n-1}\varphi_k(z)\int_0^{2\pi}i\eitheta\Phi_n'(\eitheta;\beta)\overline{\eitheta\varphi_k(\eitheta)}w(\theta)\frac{d\theta}{2\pi i}\\
&=\sum_{k=0}^{n-1}\varphi_k(z)\int_0^{2\pi}\frac{d}{d\theta}[\Phi_n(\eitheta;\beta)]\overline{\eitheta\varphi_k(\eitheta)}w(\theta)\frac{d\theta}{2\pi i}.
\end{align*}
We then integrate by parts to rewrite this as
\begin{align*}
\Phi_n'(z;\beta)&=-\sum_{k=0}^{n-1}\varphi_k(z)\int_0^{2\pi}\Phi_n(\eitheta;\beta)\frac{d}{d\theta}[\overline{\eitheta\varphi_k(\eitheta)}w(\theta)]\frac{d\theta}{2\pi i}\\
&=i\sum_{k=0}^{n-1}\varphi_k(z)\int_0^{2\pi}\Phi_n(\eitheta;\beta)\left[\overline{\eitheta\varphi_k(\eitheta)}w'(\theta)-iw(\theta)\overline{e^{2i\theta}\varphi_k'(\eitheta)+\eitheta\varphi_k(\eitheta)}\right]\frac{d\theta}{2\pi}.
\end{align*}
where we used the continuity properties of $w$.  For each $m\in\bbN$, let us define $K_m(z,t):=\sum_{j=0}^m\varphi_j(z)\overline{\varphi_j(t)}$ to be the reproducing kernel for the measure $\mu$ and polynomials of degree at most $m$.  By \cite[Section 2.2]{Wong}, we can rewrite the above expression as
\begin{align}
\nonumber\Phi_n'(z;\beta)&=i\int_{0}^{2\pi}\Phi_n(\eitheta;\beta)K_{n-1}(z,\eitheta)e^{-i\theta}w'(\theta)\frac{d\theta}{2\pi}\\
\nonumber&\qquad\qquad\qquad\qquad+\varphi_{n-1}(z)\int_{0}^{2\pi}\left(\overline{\eitheta\varphi_{n-1}(\eitheta)}+\overline{e^{2i\theta}\varphi_{n-1}'(\eitheta)}\right)\Phi_n(\eitheta;\beta)w(\theta)\frac{d\theta}{2\pi}\\
\label{s1}&=i\int_{0}^{2\pi}\Phi_n(\eitheta;\beta)K_{n-1}(z,\eitheta)e^{-i\theta}w'(\theta)\frac{d\theta}{2\pi}+n\Phi_{n-1}(z)(1-\bar{\beta}\alpha_{n-1}),
\end{align}
It follows from \cite[Section 2.3]{Wong} that (with $\zeta=\eitheta$)
\[
K_{n-1}(z,\zeta)=\frac{\kappa_{n-1}\overline{\varphi_{n-1}(\zeta)}\Phi_n(z;\beta_{\zeta})}{z-\zeta},
\qquad\mbox{ where }\qquad
\bar{\beta_{\zeta}}=\frac{\zeta\Phi_{n-1}(\zeta)}{\Phi_{n-1}^*(\zeta)}.
\]
If we plug this into (\ref{s1}) and simplify (using also the fact that $\overline{\varphi_{n-1}(\zeta)}=\zeta^{1-n}\varphi_{n-1}^*(\zeta)$ when $|\zeta|=1$), we get
\begin{align}
\nonumber\Phi_n'(z;\beta)&=\Phi_{n-1}(z)\left(n(1-\bar{\beta}\alpha_{n-1})+iz\kappa_{n-1}\int_{0}^{2\pi}\frac{\Phi_n(\eitheta;\beta)w'(\theta)\varphi_{n-1}^*(\eitheta)}{(z-\eitheta)e^{in\theta}}\frac{d\theta}{2\pi}\right)\\
\nonumber&\qquad\qquad-i\kappa_{n-1}\Phi_{n-1}^*(z)\int_{0}^{2\pi}\frac{\Phi_n(\eitheta;\beta)w'(\theta)\varphi_{n-1}(\eitheta)}{(z-\eitheta)e^{i(n-1)\theta}}\frac{d\theta}{2\pi}\\
\label{s2}&=\Phi_{n-1}(z)(n(1-\bar{\beta}\alpha_{n-1})+zG_n(z)+\bar{\beta}D_n(z))-\Phi_{n-1}^*(z)(J_n(z)-\bar{\beta}G_n(z)),
\end{align}
One can also derive (\ref{s2}) from (\ref{s1}) by using the Christoffel-Darboux formula (see \cite[Section 3]{SimonCD}) to replace $K_{n-1}(z;\eitheta)$.

The relation (\ref{s2}) holds for all values of $\beta\in\bbC$.  Let $\tau\in\bbC$ be distinct from $\beta$.  Equation (\ref{popucdef}) easily implies
\[
\Phi_{n-1}(z)=\frac{\bar{\beta}\Phi_n(z;\tau)-\bar{\tau}\Phi_n(z;\beta)}{z(\bar{\beta}-\bar{\tau})},\qquad\qquad \Phi_{n-1}^*(z)=\frac{\Phi_n(z;\beta)-\Phi_n(z;\tau)}{\bar{\tau}-\bar{\beta}}.
\]
If we plug this into (\ref{s2}), we get the following system of ODE's:
\begin{align}
\label{sys1}\Phi_n'(z;\beta)&=\Phi_n(z;\tau)\left(\frac{\bar{\beta}(n(1-\bar{\beta}\alpha_{n-1})+zG_n(z)+\bar{\beta}D_n(z))}{z(\bar{\beta}-\bar{\tau})}+\frac{(J_n(z)-\bar{\beta}G_n(z))}{\bar{\tau}-\bar{\beta}}\right)\\
\nonumber&\qquad\qquad-\Phi_n(z;\beta)\left(\frac{\bar{\tau}(n(1-\bar{\beta}\alpha_{n-1})+zG_n(z)+\bar{\beta}D_n(z))}{z(\bar{\beta}-\bar{\tau})}+\frac{(J_n(z)-\bar{\beta}G_n(z))}{\bar{\tau}-\bar{\beta}}\right)\\
\label{sys2}\Phi_n'(z;\tau)&=\Phi_n(z;\tau)\left(\frac{\bar{\beta}(n(1-\bar{\tau}\alpha_{n-1})+zG_n(z)+\bar{\tau}D_n(z))}{z(\bar{\beta}-\bar{\tau})}+\frac{(J_n(z)-\bar{\tau}G_n(z))}{\bar{\tau}-\bar{\beta}}\right)\\
\nonumber&\qquad\qquad-\Phi_n(z;\beta)\left(\frac{\bar{\tau}(n(1-\bar{\tau}\alpha_{n-1})+zG_n(z)+\bar{\tau}D_n(z))}{z(\bar{\beta}-\bar{\tau})}+\frac{(J_n(z)-\bar{\tau}G_n(z))}{\bar{\tau}-\bar{\beta}}\right).
\end{align}
This completes the proof of Theorem \ref{sysodes}.  We continue with the proof of Theorem \ref{mainode}.

\vspace{2mm}

Equation (\ref{sys1}) yields the following formula for $\Phi_n(z;\tau)$:
\begin{align}\label{tauform}
\frac{z(\bar{\beta}-\bar{\tau})\Phi_n'(z;\beta)+\Phi_n(z;\beta)\left(\bar{\tau}(n(1-\bar{\beta}\alpha_{n-1})+zG_n(z)+\bar{\beta}D_n(z))-z(J_n(z)-\bar{\beta}G_n(z))\right)}{\bar{\beta}(n(1-\bar{\beta}\alpha_{n-1})+zG_n(z)+\bar{\beta}D_n(z))-z(J_n(z)-\bar{\beta}G_n(z))}.
\end{align}
We can use (\ref{tauform}) to write
\begin{align}
\nonumber\Phi_n'(z;\tau)&=\bigg(\Phi_n''(z;\beta)z(\bar{\beta}-\bar{\tau})h_n(z;\beta;\beta)+\\
\label{tauder}&\qquad\qquad\Phi_n'(z;\beta)\left((\bar{\beta}-\bar{\tau})W[h_n(z;\beta;\beta),z]+h_n(z;\beta;\beta)h_n(z;\tau;\beta)\right)\\
\nonumber&\qquad\qquad\qquad\qquad\qquad+\Phi_n(z;\beta)W[h_n(z;\beta;\beta),h_n(z;\tau;\beta)]\bigg)h_n(z;\beta;\beta)^{-2}
\end{align}
If we substitute (\ref{tauder}) into (\ref{sys2}), then the resulting differential equation is
\begin{align}
\nonumber&\frac{z(\bar{\beta}-\bar{\tau})}{h_n(z;\beta;\beta)^2}\bigg(\Phi_n''(z;\beta)z(\bar{\beta}-\bar{\tau})h_n(z;\beta;\beta)+\Phi_n'(z;\beta)\big((\bar{\beta}-\bar{\tau})W[h_n(z;\beta;\beta),z]\\
\label{longode}&\qquad +h_n(z;\beta;\beta)h_n(z;\tau;\beta)\big)+\Phi_n(z;\beta)W[h_n(z;\beta;\beta),h_n(z;\tau;\beta)]\bigg)\\
\nonumber&\qquad=\frac{h_n(z;\beta;\tau)(z(\bar{\beta}-\bar{\tau})\Phi_n'(z;\beta)+\Phi_n(z;\beta)h_n(z;\tau;\beta))}{h_n(z;\beta;\beta)}-\Phi_n(z;\beta)h_n(z;\tau;\tau).
\end{align}
This ODE holds for all $\tau\in\partial\bbD$ (even at $\beta$ if $|\beta|=1$), so we can integrate both sides of this differential equation around $\partial\bbD$ with respect to $d\tau/(2\pi i)$.  To do so, we will use the fact that $\bar{\tau}=1/\tau$ and apply the Cauchy Integral Formula.
Performing this lengthy calculation gives the ODE:
\begin{align*}
0=&-2\bar{\beta}z^2\Phi_n''(z;\beta)+z\Phi_n'(z;\beta)\bigg[\frac{-2\bar{\beta}W[h_n(z;\beta;\beta),z]}{h_n(z;\beta;\beta)}-h_n(z;-\beta;\beta)\\
&+\bar{\beta}\left(n(1+\bar{\beta}\alpha_{n-1})-\bar{\beta}D_n(z)-zJ_n(z)/\bar{\beta}\right)\bigg]\\
&+\Phi_n(z;\beta)\bigg[\frac{-zW[h_n(z;\beta;\beta),h_n(z;-\beta;\beta)]}{h_n(z;\beta;\beta)}+h_n(z;\beta;\beta)\left[n+zG_n(z)\right]\\
&\quad-\bar{\beta}\bigg(\left(n(1-\bar{\beta}\alpha_{n-1})+zG_n(z)+\bar{\beta}D_n(z)\right)\left(n-zJ_n(z)/\bar{\beta}\right)\\
&\qquad\qquad+z(J_n(z)-\bar{\beta}G_n(z))\left(n\alpha_{n-1}-D_n(z)\right)\bigg)\bigg].
\end{align*}
After some simplification, this can be rewritten as (\ref{keyode}).
\end{proof}

In the proof of Theorem \ref{mainode}, the calculations up to equation (\ref{s2}) are very similar to those in the proof of \cite[Theorem 1.2]{IW}.  However, our approach involving the system of first order equations given by (\ref{sys1}) and (\ref{sys2}) differs from the approach used in \cite{IW}, which involves raising and lowering operators.  In the next section we will see how these differing approaches result in different differential equations (even in the case $\beta=\alpha_{n-1}$), exemplified by the fact that the differential equation we derive depends on whether one chooses $|z|<1$ or $|z|>1$ when evaluating the required integrals.

\begin{proof}[Proof of Theorem \ref{bigelec}]
First recall from \cite[Theorem 2.2.13]{OPUC1} that any collection of $n$ distinct points on the unit circle is the zero set of a paraorthogonal polynomial on the unit circle.  By Theorem \ref{BZ}, we may assume that the measure of orthogonality satisfies the hypotheses of Proposition \ref{intod}ii.

If $\mu$ is a measure that satisfies the hypotheses of Proposition \ref{intod}ii, then all of the coefficients in the differential equation in Theorem \ref{mainode} are rational functions.  Therefore, the ODE (\ref{keyode}) can be written
\begin{align}\label{hode}
0=\Phi_n''(z;\beta)+\left[\frac{1-n}{z}-\frac{h_n'(z;\beta;\beta)}{h_n(z;\beta;\beta)}\right]\Phi_n'(z;\beta)+\frac{Q_1(z)}{Q_2(z)}\Phi_n(z;\beta),
\end{align}
for some polynomials $Q_1(z)$ and $Q_2(z)$.  Since $h_n(z;\beta;\beta)$ is also a rational function, we can write it as $S_1(z)/S_2(z)$ for some polynomials $S_1$ and $S_2$.  With this notation, we can rewrite (\ref{hode}) as
\[
0=\Phi_n''(z;\beta)+\left[\frac{1-n}{z}-\frac{S_1'(z)}{S_1(z)}+\frac{S_2'(z)}{S_2(z)}\right]\Phi_n'(z;\beta)+\frac{Q_1(z)}{Q_2(z)}\Phi_n(z;\beta).
\]
If $\{k_{j,m}\}_{m=1}^{\deg(S_j)}$ are the zeros of $S_j$ ($j=1,2$), then we have
\begin{align}\label{splitform}
0=\Phi_n''(z;\beta)+\left[\frac{1-n}{z}-\sum_{m=1}^{\deg(S_1)}\frac{1}{z-k_{1,m}}+\sum_{m=1}^{\deg(S_2)}\frac{1}{z-k_{2,m}}\right]\Phi_n'(z;\beta)+\frac{Q_1(z)}{Q_2(z)}\Phi_n(z;\beta).
\end{align}
Once we show that $\Phi_n(z;\beta)$ and $Q_2(z)$ do not share any common zeros, then our discussion in Section \ref{lame} implies the desired result.

Notice that since $\mu$ satisfies the hypotheses of Proposition \ref{intod}ii, we can rewrite (\ref{hode}) as
\begin{align}\label{hode2}
0=\Phi_n''(z;\beta)+\left[\frac{1-n}{z}-\frac{h_n'(z;\beta;\beta)}{h_n(z;\beta;\beta)}\right]\Phi_n'(z;\beta)+\left[\frac{R_1(z)h_n'(z;\beta;\beta)}{h_n(z;\beta;\beta)}+R_2(z)\right]\Phi_n(z;\beta),
\end{align}
where $R_1$ and $R_2$ are rational functions without poles on the unit circle.  Now, suppose for contradiction that $Q_2$ has a zero $z_0\in\partial\bbD$ that also satisfies $\Phi_n(z_0;\beta)=0$.  It follows that $Q_1/Q_2$ has a pole at $z_0$, which implies $h_n'(z;\beta;\beta)/h_n(z;\beta;\beta)$ has a pole at $z_0$ (since $R_1$ and $R_2$ have no poles on $\partial\bbD$).  Notice that $h_n'(z;\beta;\beta)/h_n(z;\beta;\beta)$ must have a simple pole at $z_0$ and hence
\[
\left[\frac{R_1(z)h_n'(z;\beta;\beta)}{h_n(z;\beta;\beta)}+R_2(z)\right]\Phi_n(z;\beta)
\]
has a removable singularity at $z_0$.  However,
\[
\left[\frac{1-n}{z}-\frac{h_n'(z;\beta;\beta)}{h_n(z;\beta;\beta)}\right]\Phi_n'(z;\beta)
\]
has a simple pole at $z_0$ (by the Gauss-Lucas Theorem), and hence the right-hand side of (\ref{hode}) is a rational function with a pole at $z_0$, which means it is not the zero function.  This is our desired contradiction.
\end{proof}

We conclude this section by providing an explicit justification for the algorithm for finding the set of electric field generators outlined at the end of Section \ref{apps}.  From \cite[Theorem 2.2.13]{OPUC1}, we know that with $\mu_n$ and $\beta$ defined as in Step $1$ of the algorithm, the polynomial $z\Phi_{n-1}(z;\mu_n)-\bar{\beta}\Phi_{n-1}^*(z;\mu_n)$ vanishes precisely at $\{x_1,\ldots,x_n\}$.  The Bernstein-Szeg\H{o} Theorem implies $\Phi_{n-1}(z;\mu_n)=\Phi_{n-1}(z;\nu_n)$ ($\nu_n$ is defined in Step $3$ of the algorithm).  Therefore
\[
z\Phi_{n-1}(z;\mu_n)-\bar{\beta}\Phi_{n-1}^*(z;\mu_n)=\Phi_n(z;\beta;\nu_n).
\]
Since $\nu_n$ satisfies the hypotheses of Proposition \ref{intod}ii, we can make the conclusion as in Step $4$ of the algorithm, and (\ref{splitform}) shows that $\Phi_n(z;\beta;\nu_n)$ solves an ODE of the proper form to justify the placement of charges as in Step $5$.

\section{Examples}\label{exa}

In this section, we will consider applications of the results in Section \ref{apps} to specific probability measures on the unit circle.

\subsection{Example: Lebesgue Polynomials.}  Let $\mu$ be Lebesgue measure on the circle.  In this case $\alpha_{n-1}=0$ and $w$ is constant so $w'=0$.  Let us also assume $\beta=1$, so that $\Phi_n(z;\beta)=z^n-1$ and $h_n(z;\beta;\beta)=n$.  With this choice, the right-hand side of (\ref{keyode}) becomes
\begin{align*}
&n(n-1)z^{n-2}+nz^{n-1}\left(\frac{1-n}{z}\right)=0,
\end{align*}
exactly as predicted.  We see that particles of charge $+1$ located at the $n^{th}$ roots of unity are in total electrostatic equilibrium when the external field is generated by a charge of $\frac{1}{2}(1-n)$ located at the origin.

\vspace{4mm}

\subsection{Example: Degree One Bernstein-Szeg\H{o} Polynomials.}\label{example2}  In this case, let us write
\[
d\mu(\theta)=\frac{1-\left|\zeta\right|^2}{\left|1-\zeta\eitheta\right|^2}\frac{d\theta}{2\pi},\qquad\qquad|\zeta|<1.
\]
Some properties of this measure are given on \cite[page 85]{OPUC1}, but there are some typos in the information there, which we will correct.  If one sets $\zeta=r e^{i\varphi}$ and defines the Poisson kernel by
\[
P_r(x,y):=\frac{1-r^2}{1+r^2-2r\cos(x-y)}
\]
as on \cite[page 27]{OPUC1}, then it holds that $d\mu(\theta)=P_r(\theta,-\varphi)\frac{d\theta}{2\pi}$.  From this, it is easy to verify by direct computation that the Verblunsky coefficients satisfy $\alpha_0=\zeta$ and $\alpha_n=0$ for $n\geq1$.  The orthonormal polynomials $\{\varphi_n(z;\mu)\}_{n\geq0}$ for this measure satisfy
\[
\varphi_n(z;\mu)=\frac{z^n-\bar{\zeta}z^{n-1}}{\sqrt{1-|\zeta|^2}}.
\]

We can explicitly calculate the ODE satisfied by $\Phi_n(z;\beta)$, though for notational convenience, we will specialize to the case $\zeta=1/2$.  The computations are lengthy, but each step can be handled using only simple contour integration and Fourier expansions.  Indeed, if we assume $|z|>1$, then we calculate:
\begin{align*}
G_n(z)&=\frac{1}{z(2z-1)},\qquad\qquad D_n(z)=\frac{2(z^2-1)}{(2z-1)^2z^{n-1}},\qquad\qquad J_n(z)=0.
\end{align*}
With this knowledge of $G_n$, $D_n$, and $J_n$, we can write
\begin{align*}
&h_n(z;\beta;\beta)=\bar{\beta}\frac{nz^{n}(2z-1)^2+2(2z-1)z^{n}+2z\bar{\beta}(z^2-1)}{z^{n}(2z-1)^2}=:\bar{\beta}\frac{P_{1,n}(z)}{P_{2,n}(z)},
\end{align*}
where
\[
P_{1,n}(z)=z^{n}(2z-1)(n(2z-1)+2)+2z\bar{\beta}(z^2-1),\qquad P_{2,n}(z)=z^n(2z-1)^2,
\]

If we write
\[
P_{1,n}(z)=4nz\prod_{j=1}^{n+1}(z-p_{j,n}),
\]
then we can apply formula (\ref{splitform}) and write
\begin{align}\label{lame1}
0=\Phi_n''(z;\beta)+\left(\frac{2}{z-1/2}-\sum_{j=1}^{n+1}\frac{1}{z-p_{j,n}}\right)\Phi_n'(z;\beta)+\frac{Q_{1,n}(z)}{Q_{2,n}(z)}\Phi_n(z;\beta),
\end{align}
where $Q_{1,n}$ and $Q_{2,n}$ are polynomials.  Theorem \ref{bigelec} implies that if we place a particle with charge $+1$ at $1/2$ and a particle with charge $-1/2$ at each point $\{p_{j,n}\}_{j=1}^{n+1}$, then $n$ particles with charge $+1$ located at the zeros of $\Phi_n(z;\beta)$ will be in total electrostatic equilibrium.  Figure 1 is a Mathematica plot illustrating this phenomenon when $n=22$ and $\beta=-1$.

\begin{figure}[h!]\label{ex2pic}
  \centering
    \includegraphics[width=0.5\textwidth]{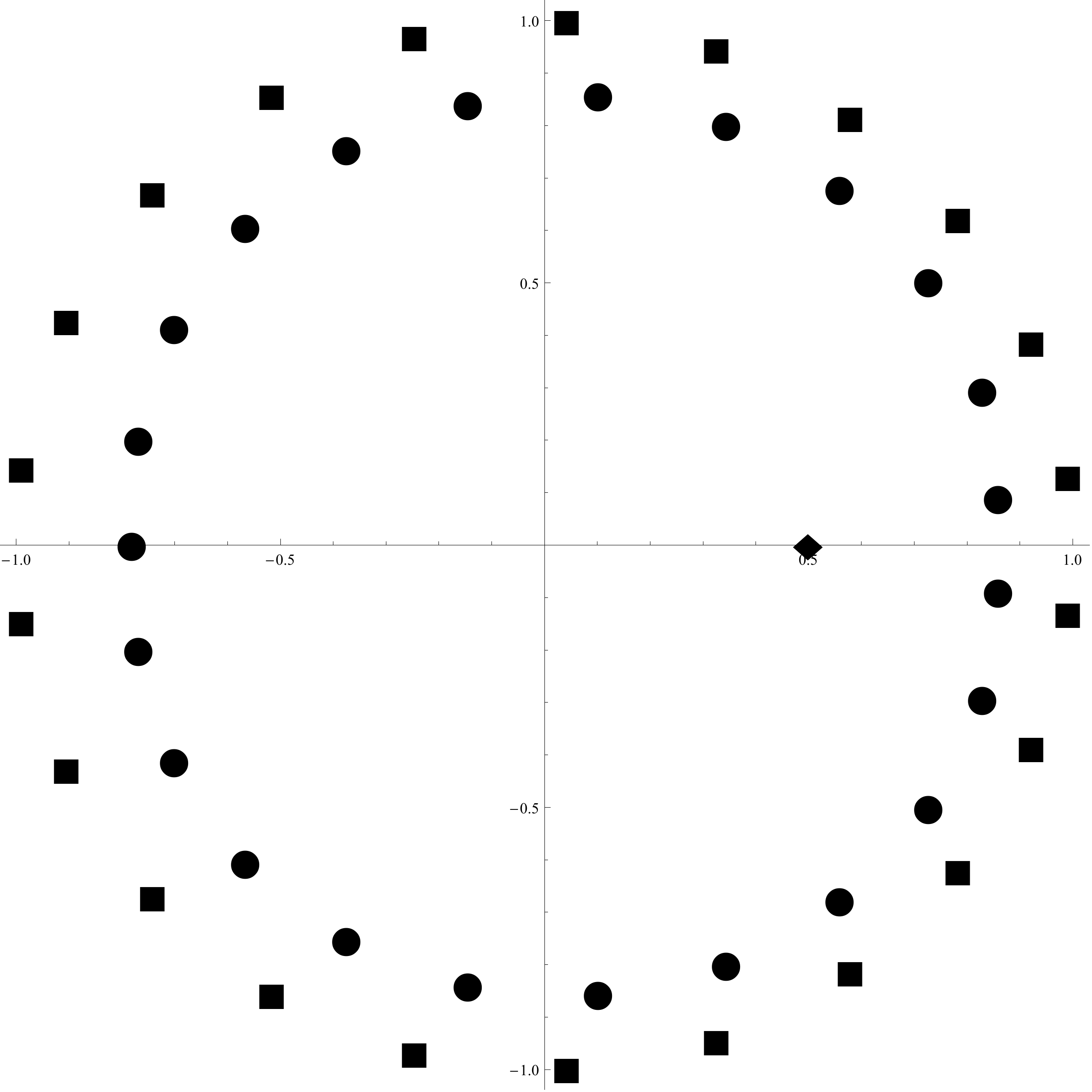}
  \caption{The squares show the location of the zeros of $\Phi_{22}(z;-1)$ and the circles show the non-zero zeros of $P_{1,22}(z)$ when $\beta=-1$.  The point $\{1/2\}$ is marked by a  diamond.}
\end{figure}

\vspace{2mm}

Notice that the functions $G_n(z)$, $D_n(z)$, and $J_n(z)$ are all holomorphic on $\bbD$ also.  If we assume $|z|<1$ and $n\geq2$, then we calculate
\begin{align*}
G_n(z)&=\frac{1}{2-z},\qquad\qquad D_n(z)=0,\qquad\qquad J_n(z)=\frac{2z^{n-2}(z^2-1)}{(z-2)^2}.
\end{align*}
With this knowledge of $G_n$, $D_n$, and $J_n$, we can write
\begin{align*}
&h_n(z;\beta;\beta)=\frac{\bar{\beta}z(z-2)(n(z-2)-2z)-2z^n(z^2-1)}{z(z-2)^2}=:\frac{P_{1,n}(z)}{P_{2,n}(z)},
\end{align*}
where
\[
P_{1,n}(z)=\bar{\beta}z(z-2)(n(z-2)-2z)-2z^n(z^2-1),\qquad P_{2,n}(z)=z(z-2)^2,
\]

If we write
\[
P_{1,n}(z)=-2z\prod_{j=1}^{n+1}(z-p_{j,n}),
\]
then we can apply formula (\ref{splitform}) and write
\begin{align}\label{lame2}
0=\Phi_n''(z;\beta)+\left(\frac{2-n}{z}+\frac{2}{z-2}-\sum_{j=1}^{n+1}\frac{1}{z-p_{j,n}}\right)\Phi_n'(z;\beta)+\frac{Q_{1,n}(z)}{Q_{2,n}(z)}\Phi_n(z;\beta),
\end{align}
where $Q_{1,n}$ and $Q_{2,n}$ are polynomials.  Reasoning as above, we conclude that particles of charge $+1$ located at the zeros of $\Phi_n(z;\beta)$ are in total electrostatic equilibrium when subject to the external field generated by a charge of $+1$ at $2$, a charge of $-1/2$ at each $p_{j,n}$, and a charge of $\frac{1}{2}(2-n)$ at the origin.  Figure 2 is a Mathematica plot illustrating this phenomenon when $n=22$ and $\beta=-1$.

\begin{figure}[h!]\label{ex2pic2}
  \centering
    \includegraphics[width=0.55\textwidth]{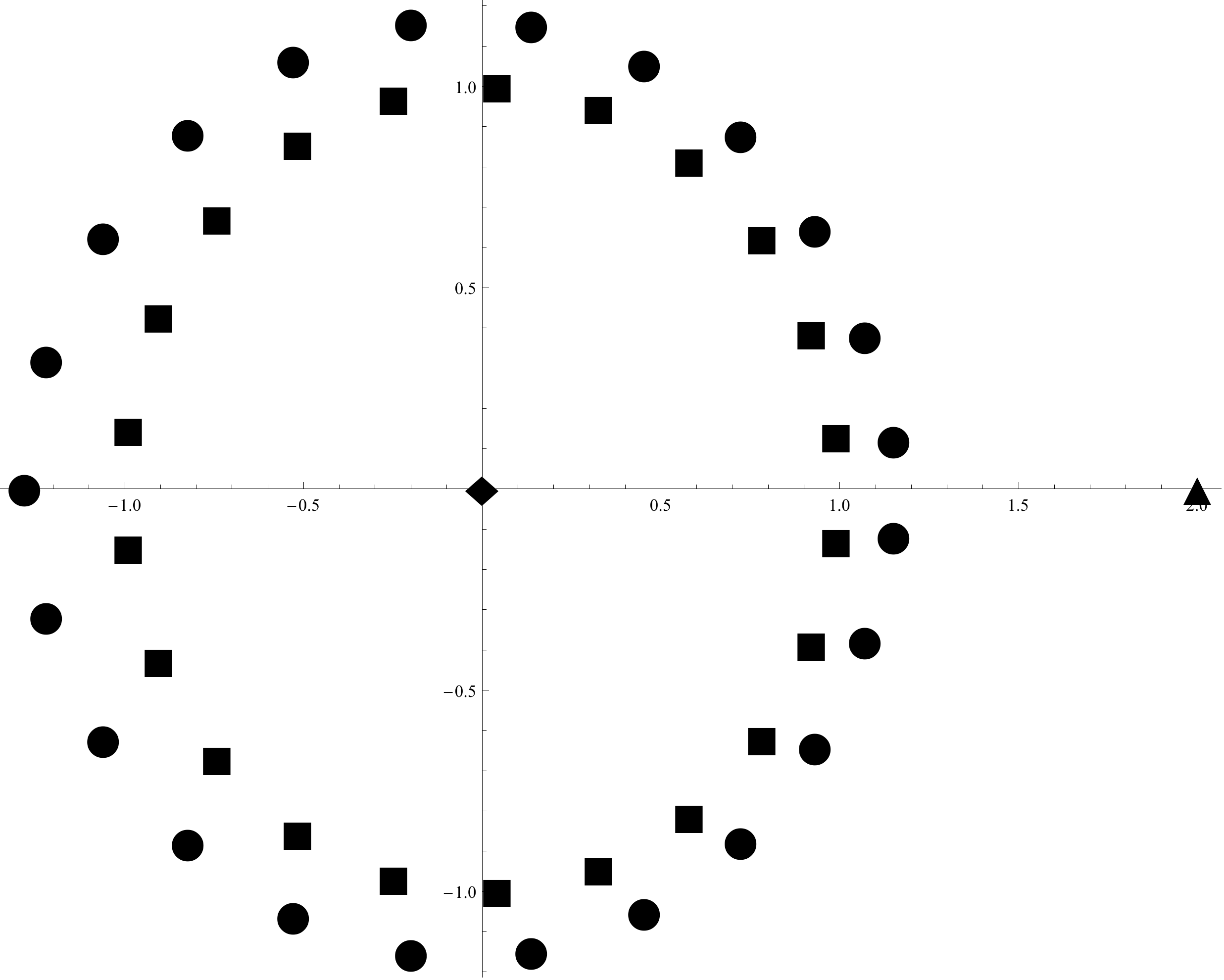}
  \caption{The squares show the location of the zeros of $\Phi_{22}(z;-1)$ and the circles show the non-zero zeros of $P_{1,22}(z)$ when $\beta=-1$.  The point $\{2\}$ is marked by a triangle and the origin is marked by a diamond.}
\end{figure}

Since a charged particle at the origin exerts a force on a charged particle on the circle that is normal to the unit circle at that point, then if we remove the charged particle at zero from this example, the zeros of $\Phi_n(z;\beta)$ are still in $\partial\bbD$-normal electrostatic equilibrium.

\vspace{4mm}

\subsection{Example: Sieved  Degree One Bernstein-Szeg\H{o} Polynomials.}  In this case, we consider measures of the form
\[
d\mu(\theta)=\frac{1-\left|\zeta\right|^2}{\left|1-\zeta e^{iM\theta}\right|^2}\frac{d\theta}{2\pi},\qquad\qquad|\zeta|<1,\quad M\in\bbN.
\]
For this measure, the Verblunsky coefficients satisfy $\alpha_{M-1}=\zeta$ and $\alpha_n=0$ for $n\neq M-1$ (see\footnote{As in the previous example, we use the table on page 85 in \cite{OPUC1} with $P_r(\theta,\varphi)$ replaced by $P_r(\theta,-\varphi)$.} \cite[page 84]{OPUC1}).  Again, for the sake of clarity, we will specialise to the case of $\zeta=1/2$ and derive the second order ODE for $\Phi_n(z;\beta)$.  The calculations are very similar to those in Example \ref{example2}, so we present fewer details here.  The only additional tricks we use are the formulas
\begin{align*}
\frac{1}{\xi(\xi^M-x)}&=\sum_{j=1}^M\frac{1}{Mx}\left(\frac{1}{\xi-x^{1/M}e^{2\pi ij/M}}-\frac{1}{\xi}\right),\\
\frac{M}{x^M-1}&=\sum_{j=1}^M\frac{1}{xe^{2\pi ij/M}-1},
\end{align*}
which are easily checked.  When $|z|>1$ and $n>M$, we have
\begin{align*}
G_n(z)=\frac{M}{z(2z^M-1)},\qquad\qquad &D_n(z)=\frac{2M(z^{2M}-1)}{(2z^M-1)^2z^{n-M}},\qquad\qquad J_n(z)=0,\\
h_n(z;\beta;\beta)&=\bar{\beta}\left(n+\frac{2M}{2z^M-1}+\frac{2M\bar{\beta}(z^{2M}-1)}{(2z^M-1)^2z^{n-M}}\right).
\end{align*}

If we perform an analysis similar to that of the previous example, then we deduce the existence of a polynomial $P_{1,n}(z)$ given by
\[
P_{1,n}(z):=Cz^M\prod_{j=1}^{n+M}(z-p_{j,n})
\]
for some $C\in\bbC$ so that $\Phi_n(z;\beta)$ satisfies
\[
\Phi_n''(z;\beta)+\left(\sum_{j=1}^M\frac{2}{z-2^{-1/M}e^{2\pi ij/M}}+\frac{1-M}{z}-\sum_{j=1}^{n+M}\frac{1}{z-p_{j,n}}\right)\Phi_n'(z;\beta)+\frac{Q_{1,n}(z)}{Q_{2,n}(z)}\Phi_n(z;\beta)=0,
\]
where $Q_{1,n}$ and $Q_{2,n}$ are polynomials.  Therefore, particles located at the zeros of $\Phi_n(z;\beta)$ - each carrying charge $+1$ - are in total electrostatic equilibrium when the external field is created by a particle of charge $+1$ at each of $\{2^{-1/M}e^{2\pi ij/M}\}_{j=1}^M$, a particle of charge $-1/2$ at each of $\{p_{j,n}\}_{j=1}^{n+M}$, and a particle of charge $\frac{1}{2}(1-M)$ at the origin.  If we remove the charge at the origin, then the zeros of $\Phi_n(z;\beta)$ are in $\partial\bbD$-normal electrostatic equilibrium.

\vspace{4mm}

\subsection{Example: Single Non-Trivial Moment.}\label{sntm}  In this case, we consider the measure given by
\[
d\mu(\theta)=(1-\cos(\theta))\frac{d\theta}{2\pi}=\frac{|1-\eitheta|^2}{2}\frac{d\theta}{2\pi}.
\]
For this measure, the Verblunsky coefficients satisfy $\alpha_{n}=-(n+2)^{-1}$ (see \cite[page 86]{OPUC1}).  The monic orthogonal polynomials are given by
\begin{align*}
\Phi_{n}(z)&=\frac{1}{n+1}\sum_{j=0}^n(j+1)z^j
\end{align*}
From this, it follows that
\[
\Phi_n(z;-1)=\frac{z^{n+1}-1}{z-1},
\]
which has zeros located at the $(n+1)^{st}$ roots of unity, except one.  This means that particles of charge $+1$ located at the zeros of $\Phi_n(z;-1)$ are in $\partial\bbD$-normal electrostatic equilibrium when the external field is generated by a single particle of charge $+1$ located at $1$ (see also \cite{FR} or \cite[Theorem 3]{BMR} with $p=1$ and $q=0$).

With the above formulas, and the fact that $\kappa_n=\sqrt{\frac{2n+2}{n+2}}$, we calculate (for $|z|>1$)
\begin{align*}
G_n(z)&=\frac{1}{n(n+1)z}\sum_{k=1}^n\frac{k^2-n-n^2}{z^k},\\
D_n(z)&=\frac{-1}{n(n+1)}\left[n^2+\sum_{k=1}^{n}\frac{n^2+2n-2kn-k^2}{z^{k}}-\frac{n^2}{z^{n+1}}\right],\\
J_n(z)&=\frac{-1}{zn(n+1)}\left[\sum_{k=0}^{n-1}\frac{(n-k)^2}{z^k}\right].
\end{align*}
With this knowledge of $G_n$, $D_n$, and $J_n$, we can write down $h_n(z;\beta;\beta)$ as a ratio of two polynomials, but the expression is extremely lengthy.  For notational convenience, we consider only the case $\beta=-1$:
\begin{align*}
&h_n(z;-1;-1)=-\frac{n(nz^{n+2}-(n+2)z^{n+1}+z+1)}{(n+1)z^{n+1}(z-1)}=:\frac{P_{1,n}(z)}{P_{2,n}(z)},
\end{align*}
where
\[
P_{1,n}(z)=-n(nz^{n+2}-(n+2)z^{n+1}+z+1),\qquad P_{2,n}(z)=(n+1)z^{n+1}(z-1).
\]

If we write
\[
P_{1,n}(z)=-n^2(z-1)\prod_{j=1}^{n+1}(z-p_{j,n}),
\]
then the differential equation of Theorem \ref{mainode} becomes
\begin{align*}
0=\Phi_n''(z;-1)+\left(\frac{2}{z}-\sum_{j=1}^{n+1}\frac{1}{z-p_{j,n}}\right)\Phi_n'(z;-1)+\frac{Q_{1,n}(z)}{Q_{2,n}(z)}\Phi_n(z;-1),
\end{align*}
where $Q_{1,n}$ and $Q_{2,n}$ are polynomials.  Reasoning as above, we conclude that particles of charge $+1$ located at the zeros of $\Phi_n(z;-1)$ are in total electrostatic equilibrium when the external field is generated by particles of charge $-1/2$ at each $p_{j,n}$ and a particle of charge $+1$ at the origin.  Figure 3 is a Mathematica plot illustrating this phenomenon when $n=14$.

\begin{figure}[h!]\label{ex4pic}
  \centering
    \includegraphics[width=0.41\textwidth]{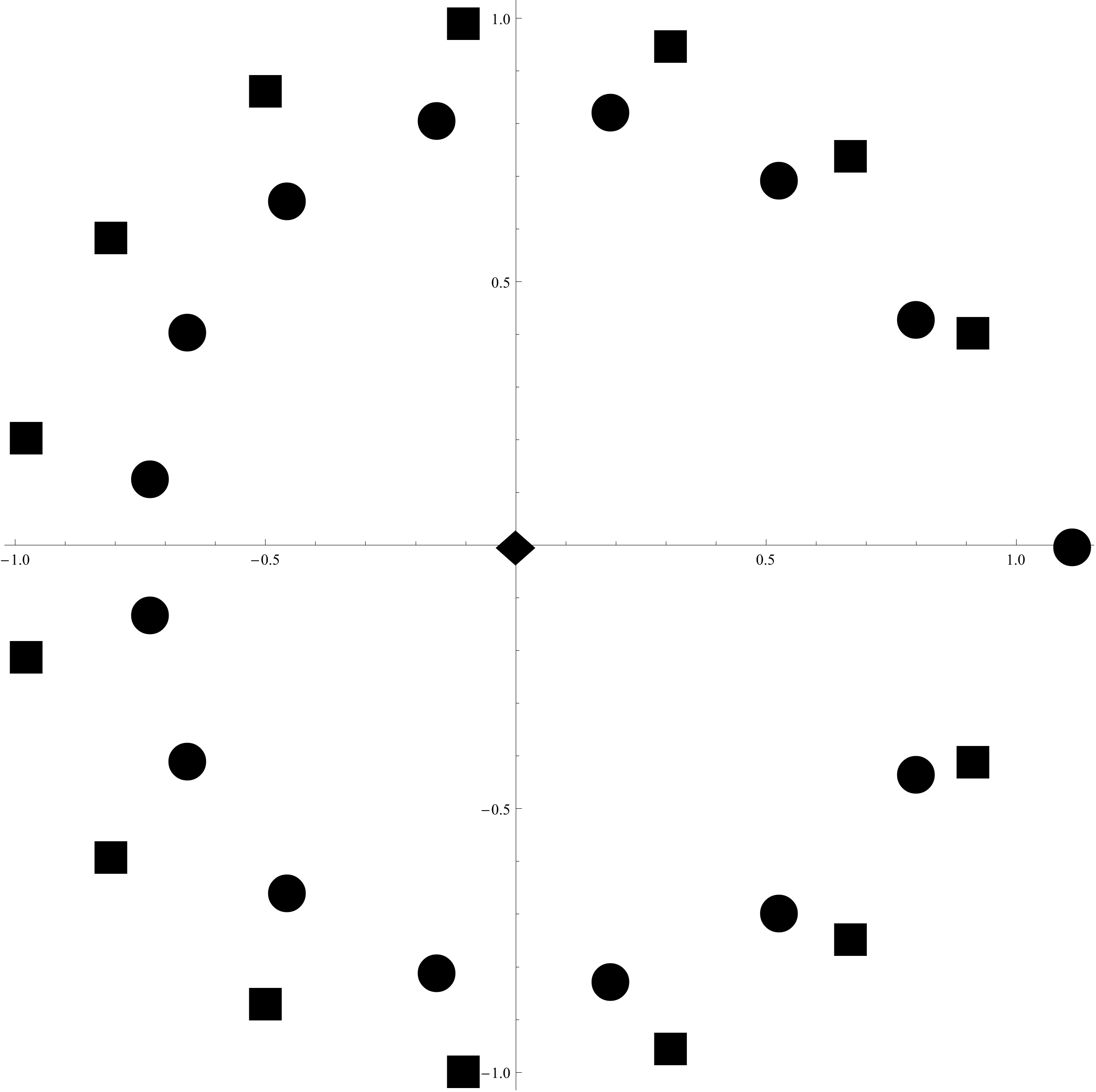}
  \caption{The squares show the location of the zeros of $\Phi_{14}(z;-1)$ and the circles show the zeros of $P_{1,14}(z)$ except $z=1$.  The diamond is located at the origin.}
\end{figure}

For the sake of comparison with the results in \cite{IW}, let us compute the differential equation from Theorem \ref{mainode} when setting $\beta=\alpha_{n-1}$ so that $\Phi_n(z;\beta)=\Phi_n(z)$.  This example was considered in \cite[Example 1]{IW} and it was shown there that the polynomial $\Phi_n(z)$ satisfies
\[
\Phi_n''(z)+\Phi_n'(z)\left(\frac{-n}{z}-\frac{3}{1-z}\right)+\Phi_n(z)\frac{2n}{z(1-z)}=0.
\]
In our calculations, we find (using Mathematica to simplify the expressions)
\[
h_n\left(z;\frac{-1}{n+1};\frac{-1}{n+1}\right)=\frac{n(z^{n+3}(n+1)^2-z^{n+2}(2n^2+6n+3)+z^{n+1}(n+2)^2-z-1)}{z^{n+1}(n+1)^3(z-1)^3}.
\]
One can then calculate
\begin{align*}
\frac{h_n'\left(z;\frac{-1}{n+1};\frac{-1}{n+1}\right)}{h_n\left(z;\frac{-1}{n+1};\frac{-1}{n+1}\right)}=\frac{1}{z}+\frac{3}{1-z}+\frac{(z^{n+2}(n+1)-z^{n+1}(n+2)+1)(z(n+1)+(n+2))}{z(z^{n+3}(n+1)^2-z^{n+2}(2n^2+6n+3)+z^{n+1}(n+2)^2-z-1)}
\end{align*}
and observe that the differential equation we derive is different than the one found in \cite{IW}.  As mentioned in Section \ref{intro}, this distinction is not unexpected.

\vspace{2mm}

If we define $G_n$, $J_n$, and $D_n$ for $z\in\bbD$, then for $n\geq2$ we have
\begin{align*}
G_n(z)&=\frac{-1}{n(n+1)}\left[\sum_{k=0}^{n-1}\left(z^k(n^2+n-(k+1)^2)\right)\right],\\
D_n(z)&=\frac{z}{n(n+1)}\left[\sum_{k=0}^{n-2}\left(z^k(k-3n+1)^2\right)\right],\\
J_n(z)&=\frac{-1}{n(n+1)}\left[n^2z^n+\sum_{k=0}^{n-1}\left(z^k((k+n+1)^2-2n-2n^2)\right)\right].
\end{align*}
With this knowledge of $G_n$, $D_n$, and $J_n$, we can write down $h_n(z;\beta;\beta)$ as a ratio of two polynomials, but the expression is extremely lengthy.  For notational convenience, we consider only the case $\beta=-1$:
\begin{align*}
h_n(z;-1;-1)=\frac{n(z^{n+2}+z^{n+1}-(n+2)z+n)}{(n+1)(z-1)}=:\frac{P_{1,n}(z)}{P_{2,n}(z)},
\end{align*}
where
\[
P_{1,n}(z)=n(z^{n+2}+z^{n+1}-(n+2)z+n),\qquad P_{2,n}(z)=(n+1)(z-1),
\]

If we write
\[
P_{1,n}(z)=n(z-1)\prod_{j=1}^{n+1}(z-p_{j,n}),
\]
then the differential equation of Theorem \ref{mainode} becomes
\begin{align*}
0=\Phi_n''(z;-1)+\left(\frac{1-n}{z}-\sum_{j=1}^{n+1}\frac{1}{z-p_{j,n}}\right)\Phi_n'(z;-1)+\frac{Q_{1,n}(z)}{Q_{2,n}(z)}\Phi_n(z;-1),
\end{align*}
where $Q_{1,n}$ and $Q_{2,n}$ are polynomials.  Reasoning as above, we conclude that particles of charge $+1$ located at the zeros of $\Phi_n(z;-1)$ are in total electrostatic equilibrium when the external field is generated by particles of charge $-1/2$ at each $p_{j,n}$ and a particle of charge $\frac{1}{2}(1-n)$ at the origin.  Figure 4 is a Mathematica plot illustrating this phenomenon when $n=14$.

\begin{figure}[h!]\label{ex4pic2}
  \centering
    \includegraphics[width=0.41\textwidth]{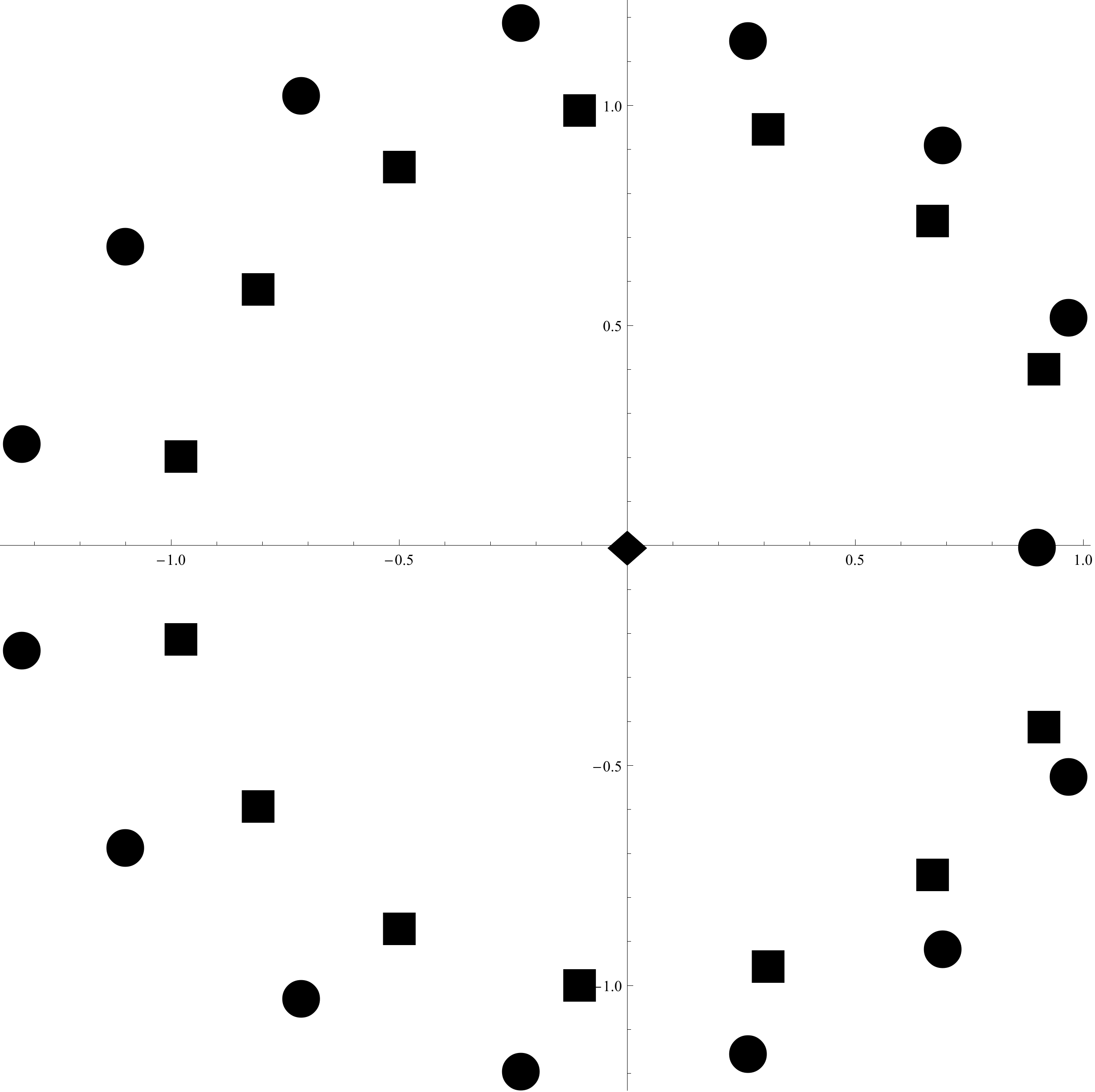}
  \caption{The squares show the location of the zeros of $\Phi_{14}(z;-1)$ and the circles show the zeros of $P_{1,14}(z)$ except $z=1$.  The diamond is located at the origin.}
\end{figure}

\newpage

\vspace{1mm}

\noindent \textsc{Brian Simanek, Baylor University Department of Mathematics}


\noindent \texttt{Brian$\_\,$Simanek$\MVAt$Baylor.edu}

\end{document}